\documentclass[letterpaper]{article} 
\usepackage{amsmath}
\usepackage[retainorgcmds]{IEEEtrantools}
\usepackage{graphicx}
\usepackage[margin=1in]{geometry}
\usepackage{amsfonts}
\usepackage{amsthm}
\usepackage{setspace}
\usepackage{mathtools}
\usepackage{listings}
\usepackage{mathabx}
\usepackage{amssymb}
\usepackage[mathscr]{euscript}
\usepackage{pictexwd,dcpic}
\newtheorem{thm}{Theorem}
\theoremstyle{definition}
\newtheorem{lem}{Lemma}

\newtheorem{crl}{Corollary}

\newcommand{\Z}{\mathbb{Z}}

\newcommand{\R}{\mathbb{R}}
\newcommand{\RP}{\mathbb{R}P}

\author{Cole Hugelmeyer} 
\title{Inscribed rectangles in a smooth Jordan curve attain at least one third of all aspect ratios. }
\begin{document} 
\maketitle

\begin{abstract}

We prove that for every smooth Jordan curve $\gamma$,  if $X$ is the set of all $r \in [0,1]$ so that there is an inscribed rectangle in $\gamma$ of aspect ratio $\tan(r\cdot \pi/4)$, then the Lebesgue measure of $X$ is at least $1/3$. To do this, we study sets of disjoint homologically nontrivial projective planes smoothly embedded in $\R\times \RP^3$. We prove that any such set of projective planes can be equipped with a natural total ordering. We then combine this total ordering with Kemperman's theorem in $S^1$ to prove that $1/3$ is a sharp lower bound on the probability that a M\"obius strip filling the $(2,1)$-torus knot in the solid torus times an interval will intersect its rotation by a uniformly random angle. 

\end{abstract}

\section{Introduction}

It is a long-standing conjecture that for every Jordan curve, there is an inscribed rectangle of every aspect ratio.  Even for smooth Jordan curves, little is known in general about the set of aspect ratios for inscribed rectangles. For instance, smooth Jordan curves always have inscribed rectangles of aspect ratio $1$ and $\sqrt{3}$, but no other aspect ratios have been proven to always be present \cite{Sch}\cite{Me}.  The question of whether every continuous Jordan curve has an inscribed square was first proposed by Toeplitz in 1911, and is a quite famous open problem. A 2014 survey of the inscribed square problem and related questions is given by Matschke \cite{survey}.  

A smooth Jordan curve is a smooth embedding $\gamma: S^1\to \R^2$, parameterizing a simple closed curve oriented counterclockwise. An inscribed rectangle of $\gamma$ is a set of four distinct points in $\R^2$ which are all in the image of $\gamma$ and which form the corners of a rectangle.  The aspect ratio of a rectangle is the ratio of its side lengths.
In this paper, we give a nontrivial lower bound for the measure of the set of aspect ratios realized by inscribed rectangles in a smooth Jordan curve:
\begin{thm}
Let $\gamma: S^1\to \R^2$ be a smooth Jordan curve, and let $X$ be the set of all $r \in [0,1]$ so that there is an inscribed rectangle in $\gamma$ of aspect ratio $\tan(r\cdot \pi/4)$. Then the Lebesgue measure of $X$ is at least $1/3$. 
\end{thm}

The proof of this theorem relies on a new topological discovery. Namely, there is a natural total ordering on any set of pairwise disjoint, homologically nontrivial projective planes smoothly embedded in $\R\times \RP^3$. This fact can be combined with Kemperman's theorem to prove Theorem 1.   In order to define this total ordering, we will need the following Lemma, which we will prove later.

Throughout this paper, we will use $\Z/2\Z$ coefficients for all homology and cohomology groups. We will often be working in the ambient space $\R\times\RP^3$. All set complements should be assumed to be complements within $\R\times \RP^3$. Thus, $Q^c =\R\times\RP^3 \setminus Q$. When we write $i_{A\hookrightarrow B}$, it means the inclusion map from $A$ to $B$. 
\begin{lem}
Let $P$ be a copy of $\RP^2$ smoothly embedded in $\R\times \RP^3$ so that $[P]$ generates $H_2(\R\times \RP^3)$. Let $\lambda$ be a large positive real number such that $P\subseteq (-\lambda,\lambda)\times \RP^3$. Then, $H^1(P^c)$ is a two-dimensional vector space over $\Z/2\Z$ for which the restriction maps to $H^1(\{\lambda\}\times \RP^3) \simeq \Z/2\Z$ and $H^1(\{-\lambda\}\times \RP^3) \simeq \Z/2\Z$ are independent. 
\end{lem}

We will now define a relation $\prec$ on disjoint homologically nontrivial projective planes in $\R\times \RP^3$. Let $P$ be a copy of $\RP^2$ smoothly embedded in $\R\times \RP^3$ so that $[P]$ generates $H_2(\R\times \RP^3)$. Let $\lambda$ be a large positive real number such that $P\subseteq (-\lambda,\lambda)\times \RP^3$, and let $i_{\lambda}:\{\lambda\}\times \RP^3\to P^c$ denote inclusion. The map $ i_{\lambda}^*:  H^1(P^c) \to H^1(\{\lambda\}\times \RP^3) \simeq \Z/2\Z$ is surjective, and $H^1(P^c)$ is two-dimensional, so the kernel of $i_{\lambda}^*$ is a 1-dimensional $\Z/2\Z$ vector space. We define $\omega_P$ to be the generator of the kernel of $i_{\lambda}^*$. From Lemma 1, we see that if $i_{-\lambda}: \{-\lambda\}\times \RP^3\to P^c$ denotes inclusion, then $i_{-\lambda}^*(\omega_P)$ is the generator of $H^1(\{-\lambda\}\times \RP^3)$. Thus, $\omega_P$ is the unique cohomology class that evaluates to zero on $\{\lambda\}\times \RP^1$ and evaluates to $1$ on $\{-\lambda\}\times \RP^1$. 
Let $P$ and $Q$ be two disjoint smoothly embedded copies of $\RP^2$ in $\R\times \RP^3$, such that $[P] =[Q]$ is the generator of $H_2(\R\times\RP^3)$. We say $P\prec Q$ if and only if $i_{P\hookrightarrow Q^c}^*(\omega_{Q})$ is the generator of $H^1(P)$. In other words, $P\prec Q$ if and only if $\omega_Q$ evaluates to $1$ on a homologically essential loop in $P$.

If $a$ and $b$ are real numbers, then $\{a\}\times \RP^2 \prec \{b\}\times \RP^2$ if and only if $a < b$. Thus, we can think of $\prec$ as being an extension of the left to right ordering in $\R$. It may be surprising that even though $\prec$ is defined on sub-manifolds of codimension 2 which may be knotted together in complicated ways, it is still a total ordering on any set of pairwise disjoint projective planes. This is the key topological fact upon which our main theorem is based.

\begin{thm}
Let $P_1,...,P_n$ be pairwise disjoint copies of $\RP^2$ smoothly embedded in $\R\times \RP^3$, so that $[P_i]$ generates $H_2(\R\times \RP^3)$ for all $i$. The reflexive relation on $\{P_1,...,P_n\}$ which extends $\prec$ is a total ordering. 
\end{thm}

In addition to the main result of this paper, this theorem can be used to give proofs of the fact that every smooth Jordan curve has an inscribed square and an inscribed rectangle of aspect ratio $\sqrt{3}$. We will prove these facts as corollaries in Section 3.

Let $T(n,m)$ be the subset of $\R^2\times S^1$ consisting of all points of the form $(\cos(\theta),\sin(\theta),e^{i\phi})$ where $\theta$ and $\phi$ are real numbers such that $ n\cdot \theta-m\cdot \phi \in 2\pi\Z$. This is called the $(n,m)$ torus link.

Consider the $S^1$ action on $ \R^2 \times S^1 \times [0,\infty)$ which comes from multiplying an element of $S^1$ with the $S^1$ coordinate. If $X\subseteq  \R^2\times S^1 \times [0,\infty)$ and $u\in S^1$, then $u\cdot X$ is defined to be $ \{(x,u \cdot y, z): (x,y,z)\in X\}$.  In this paper, we will always think of $S^1$ as the group of unit complex numbers.

We call a map $S^1\to \R^2\times S^1$ monotone if its projection to $\R^2$ is a smooth embedding. 

\begin{thm}
Let $M$ be a M\"obius strip smoothly embedded in $ \R^2 \times S^1 \times [0,\infty) $ such that $\partial M\subseteq \R^2 \times S^1 \times \{0\} $ is monotone and isotopic to $T(2,1)\times \{0\}$. Let $u$ be a uniformly random element of $S^1$. Then we have $$\mathbf{P}(M\cap\;u\cdot M \neq \varnothing) \geq 1/3$$ 
\end{thm}

\begin{thm}
There exists a M\"obius strip $M$ smoothly embedded in $ \R^2 \times S^1 \times [0,\infty) $ with $\partial M = T(2,1)\times\{0\}$ such that $M\cap (e^{i\theta})\cdot M \neq \varnothing$ if and only if $\theta\in [2\pi/3, 4\pi/3]$. Thus, the bound in Theorem 3 is tight. 
\end{thm}

\begin{proof}[Proof of Theorem 1.]
Let $\mu$ be the map $\text{Sym}^2(S^1) \to \R^2\times S^1\times [0,\infty)$ given by the piecewise formula $$ \mu\{x,y\} = \left(\frac{\gamma(x) + \gamma(y)}{2},\arg^2(\gamma(x)-\gamma(y)),\|\gamma(x)-\gamma(y)\|\right)  \;\;\;\; \text{if} \;\;\;x\neq y$$  $$  \mu\{x,x\} = \left(\gamma(x),\arg^2(\gamma'(x)),0 \right) $$ Here, $\text{Sym}^2(S^1)$ is the M\"obius strip of unordered pairs of elements of $S^1$, the norm $\|\cdot\|$ is the Euclidean norm, and $\arg^2: \R^2\setminus \{(0,0)\}\to S^1$ denotes the map such that $\arg^2(r\cos(\theta),r\sin(\theta)) = e^{i\cdot 2\theta}$, taking a vector $v$ to the point on $S^1$ with twice the angle from the $x$-axis as $v$.  

We claim that $\mu$ parameterizes a smoothly embedded M\"obius strip, $M$, and $\partial M$ is monotone and isotopic to $T(2,1)\times \{0\}$. To see that $\mu$ is an embedding, we observe that a self-intersection of the M\"obius strip would give us two unordered pairs of points on the Jordan curve with the same midpoint, the same angle of the line that passes through the two points, and the same distance from each other. This is only possible if the unordered pairs are the same in the plane, which means they are the same in $S^1$ since $\gamma$ is injective.  Taking this line of reasoning one step further, there is a diffeomorphism between the space of unordered pairs of distinct points in $\R^2$, and the space $\R^2\times S^1\times (0,\infty)$ given by $\{a,b\}\mapsto (\frac{1}{2}(a+b), \arg^2(a-b), \|a-b\|)$.  Since $\mu$ factors through this diffeomorphism away from the boundary of the M\"obius strip, we see that the differential of $\mu$ must be nonsingular on the interior of the M\"obius strip because the differential of $\gamma$ is nonsingular. The differential of $\mu$ near the boundary of the M\"obius strip projects nontrivially in both the midpoint coordinate and the distance coordinate, so it must be nonsingular at the boundary of the M\"obius strip as well. Thus, we see that $\mu$ is a smooth embedding. 

Now we must check that the boundary of $M$ is monotone and isotopic to $T(2,1)\times\{0\}$. First of all, the projection of $\partial M$ to $\R^2$ is the Jordan curve itself, so $\partial M$ is monotone.  As the Jordan curve is traversed once around, the tangent vector makes one full turn, so $\arg^2(\gamma'(x))$ makes two full traversals of $S^1$. Therefore, $\partial M$ is isotopic to $T(2,1)\times\{0\}$. We can now apply Theorem 3 to $M$ to see that for at least one third of all $u\in S^1$, we have $M \cap u\cdot M \neq \varnothing$. If $u = e^{i\theta}$, such intersections correspond to inscribed rectangles of aspect ratio $\tan(\theta/4)$, because these intersections give two pairs of points with the same midpoint, at the same distance apart, and with the angle between their line segments equal to $\theta/2$. Furthermore, $u^{-1}$ gives the same aspect ratio as $u$, so by symmetry we have at least 1/3 of the $\theta\in [0,\pi]$. Substituting $r = \theta/\pi$ gives the theorem.    \end{proof}

\section{Ordering Projective Planes}

We begin by proving Lemma 1, thus establishing that the relation $\prec$ is well-defined.
\begin{proof}[Proof of Lemma 1]
$H^1(P^c)$ is canonically isomorphic to $\text{Hom}(H_1(P^c),\Z/2\Z)$, so it suffices to prove that $H_1(P^c)$ is a two-dimensional $\Z/2\Z$ vector space with a basis consisting of the homology classes $[\{\lambda\}\times \RP^1]$ and $[\{-\lambda\}\times \RP^1]$. 

We first claim that $[\{\lambda\}\times \RP^1]$ and $[\{-\lambda\}\times \RP^1]$ are linearly independent. Since we are working over $\Z/2\Z$, it suffices to prove that they are distinct nonzero vectors. They are nonzero because they both map to the nonzero homology class in $\R\times \RP^3$ under inclusion. To see that they are distinct, take a 2-chain $\Sigma$ inside $[-\lambda,\lambda]\times \RP^3$ with $\partial \Sigma = \{\lambda\}\times \RP^1 \cup \{-\lambda\}\times \RP^1$. Without loss of generality, we may assume $\Sigma$ intersects $\{\lambda-\varepsilon\}\times \RP^2$ transversely once, for some small $\varepsilon > 0$. Since $P$ is homologous to $\{\lambda-\varepsilon\}\times \RP^2$, this implies that $\Sigma$ has nontrivial intersection with $P$ and thus cannot pass to a 2-chain in $ [-\lambda,\lambda]\times \RP^3 \setminus P$. Therefore, $[\{\lambda\}\times \RP^1]\neq[\{-\lambda\}\times \RP^1]$ in $H_1([-\lambda,\lambda]\times \RP^3\setminus P)\simeq H_1(P^c)$.

It remains to show that $\dim(H_1(P^c))\leq 2$. It suffices to prove that the kernel of the map $f: H_1(P^c)\to H_1(\R\times \RP^3)$ is at most one-dimensional. To prove this, we will show that $\ker(f)$ is generated by the homology class, $m$, represented by a meridian of a tubular neighborhood of $P$. Let $h$ be a 1-cycle representing an element of $\ker(f)$. Then, $h$ is the boundary of a 2-chain $\Sigma$ inside $\R\times \RP^3$ which intersects $P$ transversely. If we cut out a small disk from $\Sigma$ wherever it intersects $P$, we get a 2-chain $\Sigma'$ inside $P^c$ whose boundary consists of $h$, along with some number of loops which are all meridians of a tubular neighborhood of $P$. Therefore, $[h]$ is a multiple of $m$. 
\end{proof}

Let $P$ be a copy of $\RP^2$ smoothly embedded in $\R\times \RP^3$ so that $[P]$ generates $H_2(\R\times \RP^3)$, and let $N$ be a closed tubular neighborhood of $P$. Let $\pi: \partial N\to P$ be the projection from the fiber bundle structure on the closure of $N$, coming from the fact that it is a tubular neighborhood. We write $\sigma_P$ to denote the generator of $H^2(P)$. We use $N^o$ to denote the interior of $N$.

\begin{lem}
The cup product $ i_{\partial N\hookrightarrow P^c}^*(\omega_P) \cup \pi^*(\sigma_P)$ is nonzero in $H^3(\partial N)$. 
\end{lem}

\begin{proof}
Let $\alpha$ be the generator of $H^2(\R\times \RP^3)$. Let $\lambda$ be a large positive number so that $N \subseteq (-\lambda, \lambda)\times \RP^3$. Let $X = [-\lambda,\lambda]\times \RP^3\setminus N^o$. We define $\xi\in H^3(X)$ to be the cup product $i_{X\hookrightarrow P^c}^*(\omega_P)\cup i_{X\hookrightarrow \R\times\RP^3}^*(\alpha)$. Let $L$ be a 1-manifold with boundary which represents the Poincar\'e dual to $\xi$.  The parity of the number of boundary points of $L$ on a boundary component $C\subseteq \partial X$ is given by the image of $\xi$ in $H^3(C) \simeq \Z/2\Z$. Since $\omega_P$ maps to zero in $H^1(\{\lambda\}\times \RP^3)$, we know that $\xi$ maps to zero in $H^3(\{\lambda\}\times \RP^3)$. Since $\omega_P$ maps to the nonzero element of $H^1(\{\lambda\}\times \RP^3)$ and the cup product map $\cup : H^1(\RP^3)\otimes H^2(\RP^3)\to H^3(\RP^3)$ is nontrivial, we see that $\xi$ maps to the nonzero element of $H^3(\{-\lambda\}\times \RP^3)$. This tells us that $L$ has an even number of boundary points on $\{\lambda\}\times \RP^3$ and an odd number of boundary points on $\{-\lambda\}\times \RP^3$. Therefore, $L$ has an odd number of boundary points on $\partial N$. It now suffices to prove that the image of $\xi$ in $H^3(\partial N)$ is equal to the cup product $i_{\partial N\hookrightarrow P^c}^*(\omega_P) \cup \pi^*(\sigma_P)$. To prove this claim, note that $i_{P\hookrightarrow \R\times\RP^3}\circ \pi$ is homotopic to $i_{\partial N \hookrightarrow \R\times\RP^3}$, so we have $i_{\partial N\hookrightarrow \R\times\RP^3}^*(\alpha) = \pi^*(i_{P\hookrightarrow \R\times\RP^3}^*(\alpha))$. Since $P$ represents a generator of $H_2(\R\times \RP^3)$, we know that $i_{P\hookrightarrow \R\times\RP^3}^*(\alpha) = \sigma_P$. Therefore, $i_{\partial N\hookrightarrow \R\times\RP^3}^*(\alpha) = \pi^*(\sigma_P)$. We then have $$ i_{\partial N \hookrightarrow X} ^*(\xi) = i_{\partial N \hookrightarrow X} ^*( i_{X\hookrightarrow P^c}^*(\omega_P)\cup i_{X\hookrightarrow \R\times\RP^3}^*(\alpha)) = i_{\partial N\hookrightarrow P^c}^*(\omega_P)\cup i_{\partial N\hookrightarrow \R\times\RP^3}^*(\alpha) = i_{\partial N\hookrightarrow P^c}^*(\omega_P) \cup \pi^*(\sigma_P)$$ which proves our claim.
\end{proof}

Now, let $P$ and $Q$ be two disjoint smoothly embedded copies of $\RP^2$ in $\R\times \RP^3$, each of which represents the generator of $H_2(\R\times \RP^3)$. 

\begin{lem}
The relation $\prec$ is antisymmetric. That is to say, $(P\prec Q) \iff \neg(Q\prec P)$.
\end{lem}

\begin{proof}
Let $N_P$ and $N_Q$ be small closed tubular neighborhoods of $P$ and $Q$ respectively. Let $\lambda$ be a positive real number so that $N_P\cup N_Q \subseteq (-\lambda , \lambda)\times \RP^3$. Let $ X= [-\lambda,\lambda] \times \RP^3 \setminus (N_P^o\cup N_Q^o) $. Let $\omega_0$ denote the unique nontrivial element of $H^1(\R\times \RP^3)$. We define $\xi$ to be the cohomology class in $H^3(X)$ given by the three-fold cup product $\xi = i_{X\hookrightarrow P^c}^*(\omega_P)\cup i_{X\hookrightarrow Q^c}^*(\omega_Q)\cup i_{X\hookrightarrow \R\times\RP^3}^*(\omega_0)$.  Let $L$ be a 1-manifold with boundary which is Poincar\'e dual to $\xi$.  Let $n_P, n_Q, n_{-\lambda}, $ and $n_\lambda$ be the number of boundary points of $L$ on the boundary components $\partial N_P$, $\partial N_Q$, $\{-\lambda\}\times\RP^3$, and $\{\lambda\}\times \RP^3$ respectively. Then we know that $n_P+ n_Q+ n_{-\lambda}+ n_\lambda$ is even, and we also know that the parity of the number of boundary points of $L$ on a boundary component $C\subseteq \partial X$ is given by the image of $\xi$ in $H^3(C) \simeq \Z/2\Z$. We therefore know that $n_{\lambda}$ is even because $\omega_P$ and $\omega_Q$ map to zero in $\{\lambda\}\times \RP^3$ by definition. Furthermore, the cohomology classes $\omega_P$, $\omega_Q$, and $\omega_0$ all map to the nontrivial element of $H^1(\{-\lambda\}\times \RP^3)$ whose cup product cube is the nontrivial element of $H^3(\{-\lambda\}\times \RP^3)$. This means $\xi$ maps to the nontrivial element of $H^3(\{-\lambda\}\times \RP^3)$ so $n_{-\lambda}$ is odd. 

We claim that $n_P$ is odd if and only if $P\prec Q$ and $n_Q$ is odd if and only if $Q\prec P$. Without loss of generality it suffices to prove our claim just for $n_P$. The definition of $\prec$ states that $P\prec Q$ if and only if $i_{P\hookrightarrow Q^c}^*(\omega_Q) \neq 0$. By nontriviality of the cup product in $P$, have $P\prec Q$ if and only if the cup product $i_{P\hookrightarrow Q^c}^*(\omega_Q)\cup i_{P\hookrightarrow \R\times \RP^3}^*(\omega_0)$ is nontrivial. Let $1_{P\prec Q}$ be $1$ if $P\prec Q$ and $0$ otherwise. We have $$i_{P\hookrightarrow Q^c}^*(\omega_Q)\cup i_{P\hookrightarrow \R\times \RP^3}^*(\omega_0) = 1_{P\prec Q} \cdot \sigma_P$$ Since $i_{P\hookrightarrow Q^c}\circ \pi$ is homotopic to $i_{\partial N_P\hookrightarrow Q^c}$ and $i_{P\hookrightarrow \R\times\RP^3}\circ \pi$ is homotopic to $i_{\partial N_P\hookrightarrow \R\times\RP^3}$, we have $$i_{\partial N_P \hookrightarrow Q^c}^*(\omega_Q)\cup i_{\partial N_P\hookrightarrow \R\times\RP^3}^*(\omega_0)= \pi^*(i_{P\hookrightarrow Q^c}^*(\omega_Q)\cup i_{P\hookrightarrow \R\times\RP^3 }^*(\omega_0)) = 1_{P\prec Q}\cdot \pi^*(\sigma_P)  $$
And since $i_{\partial N_P \hookrightarrow Q^c}^*(\omega_Q)\cup i_{\partial N_P\hookrightarrow \R\times\RP^3}^*(\omega_0) = i^*_{\partial N_P\hookrightarrow X}(i_{X\hookrightarrow Q^c}^*(\omega_Q) \cup i_{X\hookrightarrow \R\times\RP^3}^*(\omega_0)) $, we know that $$ i^*_{\partial N_P\hookrightarrow X}(i_{X\hookrightarrow Q^c}^*(\omega_Q) \cup i_{X\hookrightarrow \R\times\RP^3}^*(\omega_0))  = 1_{P\prec Q}\cdot \pi^*(\sigma_P)$$
Now, we see that $$ i^*_{\partial N_P \hookrightarrow X}(\xi) =  i^*_{\partial N_P \to X}( i_{X\hookrightarrow P^c}^*(\omega_P)\cup i_{X\hookrightarrow Q^c}^*(\omega_Q)\cup i_{X\hookrightarrow \R\times\RP^3}^*(\omega_0))$$ $$ = i^*_{\partial N_P \hookrightarrow P^c}(\omega_P)\cup i^*_{\partial N_P\hookrightarrow X}(i_{X\hookrightarrow Q^c}^*(\omega_Q) \cup i_{X\hookrightarrow \R\times\RP^3}^*(\omega_0))  = 1_{P\prec Q}\cdot i^*_{\partial N_P \hookrightarrow P^c}(\omega_P)\cup\pi^*(\sigma_P) $$
By Lemma 2, we have $ i^*_{\partial N_P \hookrightarrow P^c}(\omega_P)\cup\pi^*(\sigma_P) = \sigma_{\partial N_P} $, where $\sigma_{\partial N_P}$ denotes the generator of $H^3(\partial N_P)$. Therefore, we have $$  i^*_{\partial N_P \hookrightarrow X}(\xi)  = 1_{P\prec Q}\cdot \sigma_{\partial N_P}$$ Since the parity of $n_p$ equals that of $i^*_{\partial N_P \hookrightarrow X}(\xi) $, and $\sigma_{\partial N_P}$ is nontrivial, this equation demonstrates our claim that $n_P$ is odd if and only if $P\prec Q$.

 Finally, since $n_\lambda$ is even, $n_{-\lambda}$ is odd, and $n_P+ n_Q+ n_{-\lambda}+ n_\lambda$ is even, we see that exactly one of $n_P$ and $n_Q$ are odd. Therefore, exactly one of the propositions $P\prec Q$ and $Q\prec P$ is true.  
\end{proof}

Suppose we have a triple $P,Q,R$ of pairwise disjoint smoothly embedded copies of $\RP^2$ in $\R\times \RP^3$, each of which represents the generator of $H_2(\R\times \RP^3)$. 

\begin{lem}
The relation $\prec$ is transitive. That is to say, if $P\prec Q$ and $Q\prec R$, then $P\prec R$
\end{lem}

\begin{proof}
Let $N_P$, $N_Q$, and $N_R$ be small closed tubular neighborhoods of $P,Q,$ and $R$ respectively. Let $\lambda$ be a positive real number so that $N_P\cup N_Q \cup N_R \subseteq (-\lambda , \lambda)\times \RP^3$. Let $ X= [-\lambda,\lambda] \times \RP^3 \setminus (N_P^o\cup N_Q^o \cup N_R^o) $. We define the cohomology class $\xi$ to be the three-fold cup product $i_{X\hookrightarrow P^c}^*(\omega_P)\cup i_{X\hookrightarrow Q^c}^*(\omega_Q)\cup i_{X\hookrightarrow R^c}^*(\omega_R)\in H^3(X)$. Let $L$ be a 1-manifold with boundary which represents the Poincar\'e dual to $\xi$. Like in the previous proof, we let $n_P,n_Q,n_R,n_{\lambda},$ and $n_{-\lambda}$ be the number of boundary points of $L$ on the respective boundary components of $X$. Again, similarly to the previous proof, we have that $n_\lambda$ is even, $n_{-\lambda}$ is odd, and the sum  $n_P+n_Q+n_R+n_{\lambda}+n_{-\lambda}$ is even. Therefore, at least one of $n_P$, $n_Q$, and $n_R$ are odd. 

We claim that $n_P$ is odd if and only if $P\prec Q$ and $P\prec R$, $n_Q$ is odd if and only if $Q\prec P$ and $Q\prec R$, and $n_R$ is odd if and only if $R\prec Q$ and $R\prec P$. Without loss of generality, it suffices to prove this claim only for $n_P$.  From the definition of $\prec$ and the nontriviality of the cup product in $\RP^2$, we see that $$i^*_{P\hookrightarrow Q^c}(\omega_Q) \cup i^*_{P\hookrightarrow R^c}(\omega_R) = 1_{P\prec Q}\cdot 1_{P\prec R} \cdot \sigma_P$$ 
Since $i_{P\hookrightarrow Q^c}\circ \pi$ is homotopic to $i_{\partial N_P\hookrightarrow Q^c}$ and $i_{P\hookrightarrow R^c}\circ \pi$ is homotopic to $i_{\partial N_P\hookrightarrow R^c}$, we have 
$$ i^*_{\partial N_P\hookrightarrow Q^c}(\omega_Q) \cup i^*_{\partial N_P\hookrightarrow R^c}(\omega_R) = 1_{P\prec Q}\cdot 1_{P\prec R} \cdot \pi^*(\sigma_P)$$
So we have $$ i^*_{\partial N_P\hookrightarrow X} ( i_{X\hookrightarrow P^c}^*(\omega_P)\cup i_{X\hookrightarrow Q^c}^*(\omega_Q) ) = 1_{P\prec Q}\cdot 1_{P\prec R} \cdot \pi^*(\sigma_P)$$
By Lemma 2, we then have $$  i^*_{\partial N_P\hookrightarrow X}(\xi) = 1_{P\prec Q}\cdot 1_{P\prec R} \cdot \sigma_{\partial N_P} $$
which proves our claim.

 We can now prove transitivity. If we assume $P\prec Q$ and $Q\prec R$, then from antisymmetry, we know that $n_Q$ and $n_R$ are even.  Therefore, $n_P$ is odd so $P\prec R$. 
 \end{proof}

\begin{proof}[Proof of Theorem 2]
Combining Lemmas 3 and 4 gives us Theorem 2. 
\end{proof}

\section{Bounding the intersection probability}

Let $Y$ be the smooth 4-manifold obtained by gluing the boundary of $\R^2\times S^1\times [0,\infty) \setminus \{0\}\times S^1\times \{0\}$ to the boundary of $(0,\infty)\times S^1\times D^2$ with the identification $f: (0,\infty)\times S^1\times S^1\to  (\R^2\setminus\{0\})\times S^1\times \{0\}$ given by the formula $$f(r,e^{i\theta}, e^{i\phi})= (r\cos(\phi),r\sin(\phi),e^{i\cdot(2\phi+\theta)},0) $$ We can think of $Y$ as a manifold obtained by capping off (2,1) torus knots with disks. 

\begin{lem}
$Y$ is diffeomorphic to $\R\times \RP^3$.
\end{lem}

\begin{proof}
If we glue together two solid tori along their boundaries by taking meridians of one to (2,1) torus knots in the other, we get $\RP^3$. We see that $(\R^2\setminus \{0\})\times S^1\times\{0\}$ is diffeomorphic to $\R\times S^1\times S^1$ because $(\R^2\setminus \{0\})$ is diffeomorphic to $\R\times S^1$. This torus bundle structure on $(\R^2\setminus \{0\})\times S^1\times\{0\}$ extends to a fibration of $\R^2\times S^1\times [0,\infty) \setminus \{0\}\times S^1\times \{0\}$ by solid tori. To see this, take the diffeomorphism $\R^2\times [0,\infty)\setminus \{0\}\times\{0\}\to (0,\infty)\times D^2$ given by $v\mapsto (\|v\|,v/\|v\|)$, where $v$ is thought of as a vector in $\R^3$ with Euclidian norm, and we identify $D^2$ with the set of points on the unit sphere with nonnegative third coordinate. Taking the product of this disk bundle structure with $S^1$ gives us our fibration by solid tori. Now, we see that $\R^2\times S^1\times [0,\infty) \setminus \{0\}\times S^1\times \{0\}$ is diffeomorphic to $\R\times S^1\times D^2$. When we glue it to $(0,\infty)\times S^1\times D^2$, which is also diffeomorphic to $\R\times S^1\times D^2$, we are taking (2,1) torus knots to meridians in each fiber, so we see that $Y\simeq \R\times \RP^3$. 
\end{proof}

Let $M_1$ and $M_2$ be disjoint M\"obius strips smoothly embedded in $\R^2\times S^1\times [0,\infty)\setminus \{0\}\times S^1\times \{0\}$ so that $\partial M_1\cup \partial M_2$ is isotopic to $T(4,2)$ in $(\R^2\setminus\{0\})\times S^1\times \{0\}$. Then we say $M_1\prec M_2$ if when we cap off $\partial M_1$ and $\partial M_2$ with disjoint, isotopically standard disks in $(0,\infty)\times S^1\times D^2$ to get disjoint smoothly embedded projective planes $P_1$ and $P_2$ in $Y\simeq \R\times \RP^3$, we have $P_1\prec P_2$.

For this definition to make sense, we must verify that capping off a M\"obius strip like this yields a projective plane which represents the generator of $H_2(Y)$. Consider the cylinder in $Y$ given by the centers of all the disks in $(0,\infty)\times S^1\times D^2$. This cylinder is a closed subset of $Y$, so it is dual to a cohomology class in $H^2(Y)$. Any projective plane $P$ obtained by capping off a M\"obius strip will intersect this cylinder transversely once, and therefore evaluates nontrivially under the cohomology class. Thus, $[P]\in H_2(Y)$ is nontrivial, and is therefore the generator of $H_2(Y)\simeq \Z/2\Z$.

\begin{lem}
Let $M_1,...,M_n$ be disjoint M\"obius strips smoothly embedded in $\R^2\times S^1 \times [0,\infty) \setminus \{0\}\times S^1\times \{0\}$ with $\partial M_1 \cup  ... \cup \partial M_n$ isotopic to $ T(2n,n)\times\{0\}$  in $(\R^2\setminus\{0\})\times S^1 \times\{0\}$. Then the reflexive relation on $\{M_1,...,M_n\}$ which extends $\prec$ is a total ordering. 
\end{lem}

\begin{proof}
If we cap off all of the M\"obius strips with disjoint disks, we get disjoint homologically nontrivial projective planes in $Y\simeq \R\times \RP^3$. Therefore, this lemma is a direct implication of Theorem 2.
\end{proof}

\begin{crl}
Let $\gamma: S^1\to \R^2$ be a smooth Jordan curve.  Then for all integers $n\geq 2$, there exists an integer $k\in \{1,...,n-1\}$ so that $\gamma$ has an inscribed rectangle of aspect ratio $\tan(\frac{\pi k}{2n})$. 
\end{crl}

\begin{proof}
As in the proof of Theorem 1, let $M$ be the M\"obius strip in $\R^2\times S^1\times [0,\infty)$ parameterized by unordered pairs of elements of $S^1$ with the piecewise formula $$ \mu\{x,y\} = \left(\frac{\gamma(x) + \gamma(y)}{2},\arg^2(\gamma(x)-\gamma(y)),\|\gamma(x)-\gamma(y)\|\right)  \;\;\;\; \text{if} \;\;\;x\neq y$$  $$  \mu\{x,x\} = \left(\gamma(x),\arg^2(\gamma'(x)),0 \right) $$ Then, $\gamma$ has an inscribed rectangle of aspect ratio $\tan(\frac{\pi k}{2n})$ for some $k\in \{1,...,n-1\}$ if any two of the M\"obius strips $M, \;e^{i\pi/n}\cdot M, \;e^{2i\pi/2}\cdot M,\; e^{3i\pi/n}\cdot M,\;...,\;e^{(n-1)i\pi/n}\cdot M$ intersect each other. We proceed by means of contradiction, and we assume all of these listed M\"obius strips are disjoint. As there is a nontrivial $\Z/n\Z$ action on these M\"obius strips, we obtain a contradiction with Lemma 6. The group action would preserve the ordering $\prec$, but a finite totally ordered set cannot admit a nontrivial order preserving group action
\end{proof}

This fact, for $n\geq 3$, was first proven in \cite{Me}, and the case of $n=2$ is the classical inscribed square problem for smooth curves. The proof in \cite{Me} relied on nonorientable slice genus bounds for torus knots in $S^3$, which were translated to nonorientable slice genus bounds for torus knots in the solid torus. These slice genus bounds acted in place of Theorem 2 to obtain the results of that paper. As we see from the following corollary, we can also use Theorem 2 to recreate the bounds that were needed for the proof in the earlier work. 

\begin{crl}
For $n\geq 2$, the torus knot $T(2n,1)\subseteq \R^2\times S^1$ does not bound a M\"obius strip in $\R^2\times S^1\times[0,\infty)$.
\end{crl}

\begin{proof}
Suppose such a M\"obius strip existed. Take the $n$-fold covering map $\phi: \R^2\times S^1\times[0,\infty)\to \R^2\times S^1\times[0,\infty)$ that raises the $S^1$ coordinate to the $n$-th power. Then $\phi^{-1}(M)$ consists of $n$ disjoint M\"obius strips with a $\Z/n\Z$ symmetry with boundary $T(2n,n)\times\{0\}$. As in the proof of Corollary 1, this is impossible.
\end{proof}

For the case of $n=2$, this bound on nonorientable slice genus is actually stronger than the one used in \cite{Me}, but it is weaker for $n>3$.

\begin{crl}
Let $\gamma: S^1\to \R^2$ be a smooth Jordan curve. Then $\gamma$ has an inscribed square and an inscribed rectangle of aspect ratio $\sqrt{3}$. 
\end{crl}

\begin{proof}
Apply Corollary 1 to $n=2$ and $n=3$ respectively. 
\end{proof}

We now will prove Theorems 3 and 4. 

\begin{proof}[Proof of Theorem 3]
The boundary $\partial M$ is monotone, so the curves of the form $u\cdot \partial M$, for $u\in S^1$, fiber an isotopically standard torus in $\R^2\times S^1\times\{0\}$. Furthermore, $\partial M$ is isotopic to $T(2,1)\times\{0\}$ so any $n$ pairwise disjoint M\"obius strips of the form $u\cdot M$ will have boundary isotopic to $T(2n,1)\times \{0\}$, fulfilling the hypothesis of Lemma 6. 

Let $X$ be the subset of $S^1$ given by $$ X = \{u\in S^1: M\cap\; u \cdot M = \varnothing \;\; \text{and} \;\; M\prec u\cdot M\} $$  Using $u^{-1}$ instead of $u$ has the effect of swapping the M\"obius strips being compared by $\prec$, so, by antisymmetry, we can conclude that $X$ is disjoint from $X^{-1}$ and $$X\cup X^{-1} = \{u\in S^1: M\cap\; u \cdot M = \varnothing \}$$  
It should be noted that $X$ is open and is therefore measurable. Now, we proceed by means of contradiction. Let $\mu$ be the Haar probability measure on $S^1$. We will assume that $\mathbf{P}(M\cap \;u\cdot M \neq \varnothing) < 1/3$ for uniformly random $u$, and thus we have that $\mu(X) > 1/3$. 

Now, we know that $\mu(X \cdot X) > 2/3$ by Kemperman's theorem  \cite{Kemperman}. (The $S^1$ case of Kemperman's theorem was first proven by Raikov \cite{Raikov}.) Since $\mu(X^{-1}) + \mu(X \cdot X) > 1$, we know that $(X\cdot X) \cap (X^{-1})$ is nonempty. Thus, we have $a,b,c\in X$ such that $a \cdot b \cdot c = 1$.  Let $M_1 = M$, $M_2 = a\cdot M$, and $M_3 = (a\cdot b)\cdot M = (c^{-1})\cdot M$. Finally, since $a,b,$ and $c$ are all in $X$, we have $M_1\prec M_2 \prec M_3 \prec M_1$. This yields a contradiction of transitivity. 
\end{proof}

\begin{proof}[Proof of Theorem 4]
Boy's surface is a particular smooth immersion of $\RP^2$ in $S^3$ \cite{Boy}\cite{Kirby}.  We will use this immersion to construct a M\"obius strip with the desired properties. First of all, we will describe Boy's surface through CW complexes. We give a CW decomposition of $\RP^2$ with three 0-cells $A,B,C$, six $1$-cells $P_{AA},P_{BB},P_{CC},$ $P_{AB},P_{BC},P_{CA}$, and four 2-cells $F_A, F_B, F_C,F_0$. See Figure 1 for a diagram. If $X$ and $Y$  are from the set $\{A,B,C\}$, the 1-cell $P_{XY}$ is a path from $X$ to $Y$, and the 2-cell $F_X$ has boundary $P_{XX}$. To finish describing the CW decomposition, we only need to describe the boundary of $F_0$. This is the path $$P_{AA}^{-1}P_{AB}P_{BC}P_{CC}^{-1}P_{CA}P_{AB}P_{BB}^{-1}P_{BC}P_{CA}$$ Boy's surface is a smooth immersion of $\RP^2$ into $S^3$ with one triple point $T$ in $S^3$, and three loops of double points going from $T$ to $T$, which we will call $L_A, L_B$, and $L_C$. The immersion function maps $P_{AA}$ and $P_{BC}$ onto $L_A$, maps $P_{BB}$ and $P_{CA}$ onto $L_B$, and maps $P_{CC}$ and $P_{AB}$ onto $L_C$. We have chosen the orientations of the paths so that these mappings are orientation preserving. Now, we will create a M\"obius strip in $\R^2\times [0,\infty)$ by excising from $S^3$ a small open ball centered at a point in the image of the interior of the face $F_0$, and then blowing up to infinity some point on the boundary of that ball that is not on our surface. This M\"obius strip has essentially the same decomposition as our projective plane, except that $F_0$ is now an annulus rather than a 2-cell, and the other boundary component of this annulus maps to a simple closed curve in $\R^2\times \{0\}$, which we may take to be the unit circle.  

\begin{figure}[h]
\caption{The CW complex in $\RP^2$ whose 1-skeleton maps to the self-intersections of Boy's surface.}
\centering
\includegraphics[scale = 0.9]{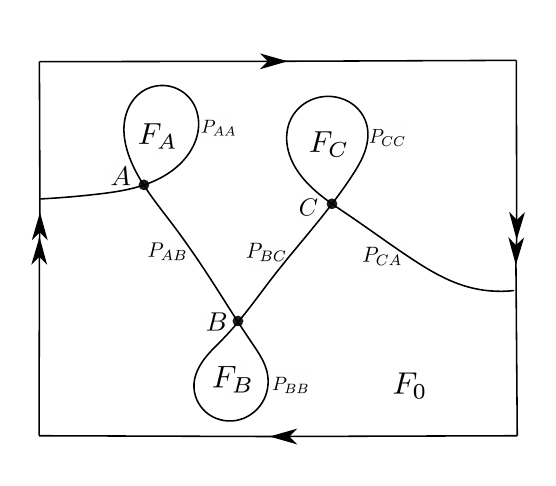}
\end{figure}

We will now assign an element of $S^1$ to every point on this M\"obius strip to describe a smoothly embedded M\"obius strip in $\R^2\times S^1\times [0,\infty)$. The points $A$, $B$, and $C$ map to $1$, $e^{2\pi i/3}$, and $e^{4\pi i/3}$ respectively. Now, for $X,Y$ from $\{A,B,C\}$, each path $P_{XY}$ simply goes in the positive direction of $S^1$ from the point in $S^1$ assigned to $X$ to the point in $S^1$ assigned to $Y$. Thus, $P_{AA}, P_{BB}, $ and $P_{CC}$ remain constant at $1$, $e^{2\pi i/3}$, and $e^{4\pi i/3}$ respectively, while the other three 1-cells each traverse exactly one third of the circle. From the way that our 1-cells map onto the loops of double points, we can now see that for any double point in our immersion, the difference in $S^1$ coordinate between the two points is in the set $\{e^{i\theta}: \theta\in [2\pi/3,4\pi/3]\}$. We can also see that it is possible to extend this $S^1$ assignment to the 2-cells $F_A, F_B,$ and $F_C$ by just making it the corresponding constant on each cell.  On the boundary of $F_0$, our $S^1$ assignment is a degree two map. Therefore, we can extend the $S^1$ assignment to the punctured $F_0$ annulus so that the boundary of our M\"obius strip is $T(2,1)\times\{0\}$. 

Finally, we can see that this M\"obius strip satisfies the criteria of Theorem 4. To see why it only intersects one third of the rotations, note that intersections between the M\"obius strip and its rotation by $u\in S^1$ will correspond to double points of the immersion with the $S^1$ coordinate difference between the two points equal to $u$.  The only such $u$ are elements of the set $\{e^{i\theta}: \theta\in [2\pi/3,4\pi/3]\}$, which only comprises one third of the circle. 
\end{proof}

\bibliography{Refrences}{}

\begin{thebibliography}{1}

\bibitem{Boy}
Werner Boy.
\newblock \"{U}ber die {C}urvatura integra und die {T}opologie geschlossener
  {F}l\"{a}chen.
\newblock {\em Math. Ann.}, 57(2):151--184, 1903.

\bibitem{Me}
C.~Hugelmeyer.
\newblock Every smooth {J}ordan curve has an inscribed rectangle of aspect
  ratio $\sqrt{3}$.
\newblock {\em arXiv 1803:07417}, 2018.

\bibitem{Kemperman}
J.~H.~B. Kemperman.
\newblock On products of sets in a locally compact group.
\newblock {\em Fund. Math.}, 56:51--68, 1964.

\bibitem{Kirby}
R.~Kirby.
\newblock What is {B}oy's surface?
\newblock {\em Notices of the AMS}, 54(10):1306--1307, 2007.

\bibitem{survey}
Benjamin Matschke.
\newblock A survey on the square peg problem.
\newblock {\em Notices Amer. Math. Soc.}, 61(4):346--352, 2014.

\bibitem{Raikov}
D.~Raikov.
\newblock On the addition of point-sets in the sense of {S}chnirelmann.
\newblock {\em Rec. Math. [Mat. Sbornik] N.S.}, 5(47):425--440, 1939.

\bibitem{Sch}
L.~G. Schnirelman.
\newblock On some geometric properties of closed curves.
\newblock {\em Usp. Mat. Nauk}, 10:34--44, 1944.

\end{thebibliography}
\bibliographystyle{plain}

\end{document}